\documentclass{article}
\usepackage[T2A]{fontenc}
\usepackage[cp1251]{inputenc}

\usepackage[tbtags]{amsmath}
\usepackage{amsfonts,amssymb,mathrsfs,amscd}

\numberwithin{equation}{section}
\newtheorem{theorem}{Theorem}[section]
\newtheorem{lemma}[theorem]{Lemma}
\newtheorem{proposition}[theorem]{Proposition}
\newtheorem{corollary}[theorem]{Corollary}

\newtheorem{proof}{Proof}

\title{Pfister forms and  a conjecture due to Colliot--Th\'{e}l\`{e}ne in the mixed characteristic case}

\author{I.A. Panin, D.N. Tyurin}
\date{}

\begin{document}
	
	\maketitle
	\begin{abstract}
	let $R$ be a regular local ring of a mixed characteristic $(0,p)$ where $p\neq 2$ is a prime number.
	Suppose that the quotient ring $R/pR$ is also regular. Fix a non-degenerate Pfister form $Q(T_{1},\ldots,T_{2^{m}})$ over $R$
	and an invertible element $c$ in $R$. Then the equation $Q(T_{1},\ldots,T_{2^{m}})=c$ has a solution over $R$
	if and only if it has a solution over the fraction field $K$.
	\end{abstract}

	\section{Introduction}
	
	Let $R$ be a local regular ring with the fraction field $K$, $Q$ be a non-degenerate quadratic form over $R$ and $c$ in $R$ be an invertible element.
%Fix some non-degenerate quadratic form $Q$ over $R$ together with some invertible element $c$ in $R$. The
In (1977, \cite{C-T})
Colliot--Th\'{e}l\`{e}ne conjectured that the equation $Q=c$ has a solution over $R$ if and only if it has a solution over the field $K$.\\\\
In (\cite{P0}, 2009) I.A. Panin proved this conjecture for the case when $R$ contains the field of rational numbers. That proof was based
on a moving lemma for Levine--Morel's algebraic cobordism. Later, together with K.I. Pimenov he extended this result to the case when $R$
contains an arbitrary infinite field (\cite{PP1}, \cite{PP2}). The proof in \cite{PP1} still used a sofisticated tool. Namely,
it used the Gabber version of the de Jong theorem on alteration. The proof in \cite{PP2} is quite "elementary". It could be discovered
right after the paper \cite{OP}. However in all these papers the base field is supposed to be infinite.\\\\
Only in Scully's paper (2018,\cite{Sc}) the conjecture was proved for all regular local ring containing a field
(particularly, containing a finite field).
{\it Thus, the conjecture is solved in positive for all $R$ containing a field.
}\\\\
What is known about the conjecture in the mixed characteristic case?
Very little.
There is a resent Panin's result in \cite[Section 7]{P4}. However these results are obtained
under rather strong assumptions on $R$. Namely, $R$ is the local ring of an $A$-smooth scheme $X$
satisfying the $A$-relative {\it Noether Normalization Lemma}.\\\\
Also there is \cite[Corollary 3.3]{P5} where the conjecture is settled for any unramified $R$ under the assumption 
that the form 
$Q$ is a two-fold Pfister form.\\\\
{\it The goal of the present preprint paper is to extend the mentioned result
\cite[Corollary 3.3]{P5} to arbitrary Pfister forms (in the mixed characteristic case)}. \\\\
%%%result even further to the case of $R$ having mixed characteristic.
Our main result is the following theorem:
let $R$ be a regular local ring of a mixed characteristic $(0,p)$ where $p\neq 2$ is a prime number.
Suppose that the quotient ring $R/pR$ is also regular. Fix a non-degenerate Pfister form $Q(T_{1},\ldots,T_{2^{m}})$ over $R$
and an invertible element $c$ in $R$. Then the equation $Q(T_{1},\ldots,T_{2^{m}})=c$ has a solution over $R$
if and only if it has a solution over the fraction field $K$. 
This result is proved in the present preprint under the assumption that the residue field is infinite. 
However this result 
{\it remains true even without this assumption}. 
The latter will be proved in our next preprint using 
the main result due to Scully [Sc]. \\\\
Here is a sketch of the proof: assume that the equation $Q=c$ has a solution over $K$. It yields that for certain auxiliary Pfister form
$\widetilde{Q}$ over $R$, constructed from pair $(Q,c)$, is isotopic over $K$. Since $\widetilde{Q}_{K}$ is a Pfister form
this yields it is hyperbolic over $K$. According to the recent result of Cesnavicius \cite[Corollary 9.6]{C}.
the form $\widetilde{Q}$ is also hyperbolic over $R$. Using now the key geometric theorem formulated below, we conclude the proof.\\\\
Here is the mentioned 
{\it key geometric theorem}: 
let $M$ be a free module over a local ring $R$ of rank $2n$ with a fixed non-degenerate quadratic form $Q$
and let $N$ be its free direct summand of rank $n+1$ such that the restriction $Q|_{N}$ is also non-degenerate.
Then $N$ containes a unimodular isotopic vector. The most basic case of this theorem discussed in section~\ref{sec:vect}
is when $R$ coincides with some algebraically closed field $k$. Generally by $k$ we denote the residue field $R/\mathfrak{m}R$ of a local ring $R$.\\\\
All over the preprint the characteristic of various fields $k$ is not 2.\\\\
The paper is organized as follows. Sections~\ref{sec:vect} and~\ref{sec:mod} contain
the proof of the key geometric theorem for the cases of a vector space over the residue field $k:=R/\mathfrak{m}R$ of a local ring $R$
and a free module over the local ring $R$ respectively. Section~\ref{sec:main} contains the proof of the main result.\\\\

The authors thank the excellent environment of the International Mathematical Center at POMI. This work is supported by Ministry of Science and Higher Education of the Russian Federation, agreement № $075-15-2022-289$.

	\section{The case of a vector space}\label{sec:vect}
	
	Let $V$ be a vector space of dimension $n\in\mathbb{N}$ over an algebraicaly closed field $k$ ($char(k)\neq2$). 
For each $j\leqslant n$ denote by $X^{V}_{j}$ the projective variety of complete flags of length $j$ in $V$. In particular, for $j=1$ this variety  coincides with the projective space $\mathbb{P}(V)$ and therefore has dimension $n-1$. There exists a natural surjective morphism
	$$
	\pi_{j}=X^{V}_{j}\twoheadrightarrow X^{V}_{j-1}\,, \, \, \{L_{j}\supset\ldots\supset L_{1}\}\to\{L_{j-1}\supset\ldots\supset L_{1}\}\,.
	$$
Note that its fibre over a point $\widetilde{L}:=\{L_{j-1}\supset\ldots\supset L_{1}\}\in X^{V}_{j-1}$ is naturally isomorphic to the projectivisation $\mathbb{P}(V/L_{j-1})$ of the quotient vector space $V/L_{j-1}$ of dimension $n-j+1$. Thus the dimention of $\pi^{-1}(\widetilde{L})$ is equal to $n-j$. Using theorem of dimention of fibres and applying the induction by $j$ we get the following equality:
\begin{equation}\label{eq:dimX}
	dim(X^{V}_{j})=(n-1)+\ldots+(n-j)=nj-\dfrac{j(j+1)}{2}\,.
\end{equation}

Now suppose that $n\geqslant2$ and fix some non-degenerate quadratic form $Q$ on $V$. We say that a flag is totally isotopic if each of its subspaces is a totally isotopic subspace over $Q$ (note that such a subspace cannot be more than $\Large[\frac{n}{2}\Large]$-dimensional). We denote by $\mathring{X}^{V}_{j}$ the projective variety of complete totally isotopic flags of length $j$. There is an obvious closed embedding $\mathring{X}^{V}_{j}\hookrightarrow X^{V}_{j}$ that in case of $j=1$ associates $\mathring{X}^{V}_{1}$ with a hyperplane in $X^{V}_{1}\cong\mathbb{P}(V)$ defined by the equation $Q(v)=0$. Consequently $\mathring{X}^{V}_{1}$ has dimension $n-2$.

As in case of $X^{V}_{j}$  here existis a natural morphism
$$
\tau_{j}=\mathring{X}^{V}_{j}\twoheadrightarrow\mathring{X}^{V}_{j-1}\,, \, \, \{L_{j}\supset\ldots\supset L_{1}\}\to\{L_{j-1}\supset\ldots\supset L_{1}\}\,,
$$

 which is also surjective according to the Witt's theorem. Its fibre over a point $\widetilde{L}:=\{L_{j-1}\supset\ldots\supset L_{1}\}$ is isomorphic to the projective variety of all totally isotopic $j$-dimensional subspaces of $V$ that contain $L_{j-1}$. By definition any subspace of such kind belongs to $L_{j-1}^{\perp}$, therefore one can identify the $\tau_{j}^{-1}(L)$ with a hyperplane in $\mathbb{P}(L_{j-1}^{\perp}/L_{j-1})$ defined by the equation $\overline{Q}(v)=0$. Thus the dimention of $\tau_{j}^{-1}(L)$ is equal to $n-2(j-1)-2=n-2j$.  Using theorem of dimension of fibres and applying the induction by $j$ we get the following equality:
\begin{equation}\label{eq:dimXo}
	dim(\mathring{X}^{V}_{j})=(n-2)+\dots+(n-2j)=nj-j(j+1)\,.
\end{equation}

Finally denote by $\mathring{Y}^{V}_{j}$ the projective variety of all totally isotopic $j$-dimensional subspaces in $V$. One can define a surjective morphism
$$
\gamma_{j}=\mathring{X}^{V}_{j}\twoheadrightarrow\mathring{Y}^{V}_{j}\,, \, \, \{L_{j}\supset\ldots\supset L_{1}\}\to L_{j}\,,
$$
which fibre over a totally isotopic subspace $L$ in $V$ is naturally isomorphic to $X^{L}_{j-1}$. Thus, according to formula~\ref{eq:dimX} there is an equality
$$
dim(X^{L}_{j-1})=j(j-1)-\dfrac{j(j-1)}{2}=\dfrac{j(j-1)}{2}\,.
$$
Applying theorem of dimension of fibres we get an equality
\begin{equation}\label{eq:dimYo}
	dim(\mathring{Y}^{V}_{j})=nj-j(j+1)-\dfrac{j(j-1)}{2}\,.
\end{equation}

\begin{proposition}\label{prop:vect}
Let $k$ be an algebraically closed field.
Let  $V$ be a $2n$-dimensional $k$-vector space equipped with a hyperbolic form $Q$. 
Let $P\subset V$ be a $n+1$-dimensional subspace
%%%Suppose that there is some $n+1$-dimensional subspace 
such that the restriction $Q|_{P}$ is non-degenerate. Then there exists a totally isotopic $n$-dimensional subspace $W\subset V$ such that $dim(P\cap W)=1$.
\end{proposition}

\begin{proof}
	According to formula ~\ref{eq:dimYo} the dimension of the projective variety $\mathring{Y}^{V}_{n}$ of all totally isotopic $n$-dimensional subspaces $W\subset V$ is equal to
	\begin{equation}\label{eq:dimYon}
		dim(\mathring{Y}^{V}_{n})=2n^{2}-n(n+1)- \dfrac{n(n-1)}{2}=n^{2}-n-\dfrac{n^{2}-n}{2}=\dfrac{n^{2}-n}{2}\,.
	\end{equation}	
	For any $j\leqslant n$ denote by $Z_{j}$ the projective variety of all totally isotopic $n$-dimensional subspaces $W\subset V$ such that $dim(W\cap P)=j$. There is an obvious disjoint decomposition
	$$
	\mathring{Y}^{V}_{n}=Z_{1}\sqcup\dots\sqcup Z_{\large[\frac{n+1}{2}\large]}\,.
	$$
	For any $j\leqslant\large[\frac{n+1}{2}\large]$ there exists a natural morphism
	$$
	\lambda_{j}:Z_{j}\rightarrow\mathring{Y}^{P}_{j}\,,\,\,W\to P\cap W\,,
	$$
	which is surjective according to Witt's theorem. Its fibre $\lambda_{j}^{-1}(U)$ over a point $U\in\mathring{Y}^{P}_{j}$ coincides with a set
	$$
	\{W\subset V|\, dim(W)=n,\, Q|_{W}\equiv0,\, W\cap P=U\}\,,
	$$
	and therefore can be identified with a projective subvariety of
	$$
	\mathring{Y}^{U^{\perp}/U}_{n-1}=\{W\subset V|\, dim(W)=n,\, Q|_{W}\equiv0,\, U\subset W\}\,.
	$$
	Since $dim(U^{\perp}/U)=2(n-j)$ by applying formula~\ref{eq:dimYon} we get an equality
	$$
	dim\big(\mathring{Y}^{U^{\perp}/U}_{n-1}\big)= \dfrac{(n-j)^{2}-(n-j)}{2}\,.
	$$
	
	Thus, theorem of dimension of fibres gives us an estimate
	$$
	dim(Z_{j})\leqslant dim(\mathring{Y}^{P}_{j})+\dfrac{(n-j)^{2}-(n-j)}{2}\,.
	$$
	
	Finally, using formula~\ref{eq:dimYo} we get an estimate
	$$
	dim(Z_{j})\leqslant(n+1)j-j(j+1)-\dfrac{j(j-1)}{2}+\dfrac{(n-j)^{2}-(n-j)}{2}\,
	$$
	whose right hand side is equal to
	$$
	\dfrac{n^{2}-n}{2}-(j^{2}-j)=dim(\mathring{Y}^{V}_{n})-(j^{2}-j)\,.
	$$
	Therefore, if $j\geqslant2$ then there is a strict inequality $dim(Z_{j})<dim(\mathring{Y}^{V}_{n})$. In particular the subset
	$$
	Z_{1}=\mathring{Y}^{V}_{n}\backslash\big(Z_{2}\sqcup\dots\sqcup Z_{\large[\frac{n+1}{2}\large]}\big)
	$$
	is open and non-empty in $\mathring{Y}^{V}_{n}$. This concludes the proof.
\end{proof}

\begin{corollary}\label{cor:nonclosed}
Let $k$ be an infinite field.
Let  $V$ be a $2n$-dimensional $k$-vector space equipped with a hyperbolic form $Q$.
Let $P\subset V$ be a $n+1$-dimensional subspace
%%%Suppose that there is some $n+1$-dimensional subspace
such that the restriction $Q|_{P}$ is non-degenerate. Then there exists a totally isotopic $n$-dimensional subspace $W\subset V$ such that $dim(P\cap W)=1$.
\end{corollary}

\begin{proof}
The variety $\mathring{Y}^{V}_{n}$ is $k$-rational. Its open subvariety  
$Z_{1}$ is non-empty since it is non-empty over the algebraic closure of $k$. 
Since the field $k$ is infinite the set of $k$-rational points of $Z_{1}$ is non-empty.
This proves the corollary.
\end{proof}

\section{Case of a free module over a local ring}\label{sec:mod}

During this section by $R$ and $k$ we will denote a regular local ring and its residue field accordingly. Recall that $char(k)\neq2$.

\begin{lemma}\label{lem:matinv}
	A matrix $A\in Mat_{n\times n}(R)$ has an inverse if and only if the quotient mathrix $\overline{A}\in Mat_{n\times n}(k)$ has an inverse.
\end{lemma}
\begin{proof}
	One implication is trivial. Suppose that $\overline{A}$ is invertible. It means that $det(\overline{A})$ is not equal to zero in $k$, which in turn means that  $det(A)$ does not belong to the maximal ideal of $R$. Thus $det(A)$ is invertible which means that $A$ has an inverse.
\end{proof}

\begin{lemma}\label{lem:freemod}
	Let $M$ be a free $R$-module of rank $n$. Let $\overline{u}_{1},\dots,\overline{u}_{r}$ be a set of lineary independent elements of $k$-linear vector space $\overline{M}$ and fix some arbitrary liftings $u_{i}\in M$ for each $\overline{u}_{i}$. Then the submodule $U:=\langle u_{1},\ldots,u_{r}\rangle_{R}$ of $M$ is a free direct summand of $M$ of rank $r$.
\end{lemma}
\begin{proof}
	Choose a basis of the type $\overline{u}_{1},\dots,\overline{u}_{r},\overline{v}_{1},\dots,\overline{v}_{n-r}$ in $\overline{M}$ and fix some liftings $v_{j}\in M$ for each $\overline{v}_{j}$. Then the elements $\{u_{i},v_{j}\}$ generate $M$ as an $R$-module according to Nakayama's lemma. Denote by $A$ the corresponding transformation matrix. Since $\{\overline{u}_{i},\overline{v}_{j}\}$ is a basis in $\overline{M}$, the quotient matrix $\overline{A}$ has an inverse. It follows from lemma~\ref{lem:matinv} that the matrix $A$ has an inverse as well. Thus $\{u_{i},v_{j}\}$ is also a basis in $M$. This concludes the proof.
\end{proof}

\begin{corollary}\label{cor:freemod}
	Any direct summand $U$ of a free module $R^{n}$ is also free. Moreover if $\{u_{1},\dots,u_{r}\}$ is a basis in $U$ then $\{\overline{u}_{1},\dots,\overline{u}_{r}\}$ is a basis in $\overline{U}$ and vice versa.
\end{corollary}
\begin{proof}
	Choose a projection $R^{n}\twoheadrightarrow U$. By definition it is surjective and its composition $U\hookrightarrow R^{n}\twoheadrightarrow U$ with the natural embedding $U\hookrightarrow R^{n}$
coincides with the identity automorphism $Id_{U}$. Reducing modulo maximal ideal $\mathfrak{m}\subset R$ choose a basis $\{\overline{u}_{1},\dots,\overline{u}_{r}\}$ in the corresponding subspace $\overline{U}$ of $k^{n}$ and fix some lifting $\{u_{1},\dots,u_{r}\}$ that lies in the image of $U\hookrightarrow R^{n}$. Denote by $U^{'}$ the submodule $\langle u_{1},\dots,u_{r}\rangle_{R}$ in $R^{n}$  It follows from Nakayama's lemma, that the morphsim $U^{'}\longrightarrow U$ induced by the projection $R^{n}\twoheadrightarrow U$ is surjective, which in turn means that $U$ and $U^{'}$ actually coincide as $R$-submodules in $R^{n}$. It follows from lemma~\ref{lem:freemod} that $U^{'}$ is a free module with a basis $\{u_{1},\dots,u_{r}\}$.
	
	Likewise, if $U$ is a direct summand of $R^{n}$ and $\{u_{1},\dots,u_{r}\}$ is its basis then there exists a surjective $R$-linear morphism $R^{n}\twoheadrightarrow R^{r}$, that takes $u_{i}$ to the $i$-th element of the standard basis of $R^{r}$. Since the associated quotient morphism $k^{n}\twoheadrightarrow k^{r}$ of $k$-linear vector spaces stays surjective, $\{\overline{u}_{1},\dots,\overline{u}_{r}\}$ is obviously a basis in $\overline{U}$.
\end{proof}

Before introducing the main geometric result of this section we will prove a small technical lemma that will play a cruical part in the proof:
\begin{lemma}\label{lem:orthogonal}
	Let $M$ be a free $R$-module with a fixed non-degenerate quadratic form $Q$. Denote by $O_{R}(M)$ the group of $R$-linear orthogonal automorphisms of $M$. Then the natural homomorphism $O_{R}(M)\rightarrow O_{k}(\overline{M})$ is surjective.
\end{lemma}
\begin{proof}
	For our convenience we will also denote by $Q$ the corresponding bilinear form. Note that the group $O_{k}(\overline{M})$ is generated by reflections of the type
	$$
	r_{\overline{u}}:\overline{M}\rightarrow\overline{M}\,,\,\, \overline{v}\to\overline{v}-\dfrac{\overline{Q}(\overline{u},\overline{v})}{\overline{Q}(\overline{u},\overline{u})}\overline{u}
	$$
	where $\overline{u}$ goes over all elements of $\overline{M}$ such that $\overline{Q}(\overline{u},\overline{u})\neq0$. In particular this means that for any such $\overline{u}$ and for any of its liftings $u\mapsto\overline{u}$ in $M$ the corresponding element $Q(u,u)$ does not belong to the maximal ideal of $R$ and therefore is invertible. Thus there is a well defined reflection
	$$
	r_{u}:M\rightarrow M\,,\,\, v\to v-\dfrac{Q(u,v)}{Q(u,u)}u\,,
	$$
	
	that belongs to $O_{R}(M)$ and is obviously a lifting of the reflection $r_{\overline{u}}$. This concludes the proof.
\end{proof}

Now let $M$ be a free $R$-module of rank $2n$, with a fixed hyperbolic form $Q$ on it. Suppose that $N\subset M$ is a free direct summand of $M$ such that $rk(N)=n+1$ and the restriction $Q|_{N}$ is non-degenerate. Then the corresponding $k$-linear vector spaces $\overline{M}$ and $\overline{N}$ satisfy the conditions of corollary~\ref{cor:nonclosed}. Indeed if $\{x_{1},\ldots,x_{n+1}\}$ is some basis of $N$ then it follows from corollary~\ref{cor:freemod} that $\{\overline{x}_{1},\ldots,\overline{x}_{n+1}\}$ is a basis of $\overline{N}\subset\overline{M}$. Since the restriction of the form $Q$  on $N$ is non-degenerate the corresponding morphism $Q|_{N}^{\vee}:N\to N^{\vee}$ of $R$-modules is an isomorphism. Therefore the restriction $\overline{Q}|_{\overline{N}}$ is non-degenerate according to lemma~\ref{lem:matinv}. Finally it is obvious that the form $\overline{Q}$ on $\overline{M}$ is also hyperbolic.
Thus, according to corollary~\ref{cor:nonclosed} there exists a totally isotopic $n$-dimensional subspace $\overline{W}\subset\overline{M}$ such that $dim(\overline{N}\cap\overline{W})=1$.

\begin{proposition}\label{prop:mod} Let $M$ be a free $R$-module of rank $2n$, with a fixed hyperbolic form $Q$ on it. Suppose that $N\subset M$ is a free direct summand of $M$ such that $rk(N)=n+1$ and the restriction $Q|_{N}$ is non-degenerate. Then there exists an $R$-submodule $W$ in $M$ such that:
	\begin{itemize}
		\item $W$ is a free direct summand of $M$ of rank $n$;
		\item the quotient morphism  $M\twoheadrightarrow\overline{M}$ surjectively maps $W$ to $\overline{W}$;
		\item $Q|_{W}\equiv0$.
	\end{itemize}
	Additionally, the submodule $W\cap N$ of $N$ will be a free direct summand of rank $1$.
\end{proposition}
\begin{proof} Denote by $\mathbb{H}_{R}$ the hyperbolic plane with a standard basis $\{e,f\}$.
	Since $Q$ is hyperbolic there existis a decomposition $M=\mathbb{H}_{R}\oplus\ldots\oplus\mathbb{H}_{R}$. Denote by $\{e_{1},f_{1},\ldots,e_{n},f_{n}\}$ the corresponding basis in $M$. Then the $R$-submodule $\langle e_{1},\ldots,e_{n}\rangle_{R}\subset M$ is obviously a free totally isotopic summand of $M$ of rank $n$. Accordingly the set  $\{\overline{e}_{1},\overline{f}_{1},\ldots,\overline{e}_{n},\overline{f}_{n}\}$ is a basis of $k$-linear vector space $\overline{M}=\mathbb{H}_{k}\oplus\ldots\oplus\mathbb{H}_{k}$ and the subspace $\langle\overline{e}_{1},\ldots,\overline{e}_{n}\rangle_{k}\subset \overline{M}$ is totally isotopic and has dimension $n$. Choose some basis $\{\overline{w}_{1},\ldots,\overline{w}_{n}\}$ in $\overline{W}$. According to Witt's theorem there exists some $\overline{A}\in O_{k}(\overline{M})$, such that $\overline{A}(\overline{e_{i}})=\overline{w_{i}}$ for any $1\leqslant i\leqslant n$. It follows from lemma~\ref{lem:orthogonal} that $\overline{A}$ has some lifting $A\in O_{R}(M)$. Put $w_{i}:=Ae_{i}$ for $1\leqslant i\leqslant n$. Then the submodule $W:=\langle w_{1},\dots,w_{n}\rangle_{R}\subset M$ obviously satisfies the conditions of proposition~\ref{prop:mod}.
	
	Now denote by $\pi:N\rightarrow M/W$ the composition of the natural embedding $N\hookrightarrow M$ and the quotient homomorphism $M\twoheadrightarrow M/N$. Reducing modulo maximal ideal $\mathfrak{m}\subset R$ we get the $k$-linear morphism $\overline{\pi}:\overline{N}\rightarrow\overline{M}/\overline{W}$. Its kernel coincides with the subspace $\overline{N}\cap\overline{W}\subset\overline{N}$. It immediately follows from the dimension theorem for vector spaces that $\overline{\pi}$ is surjective. It then follows from Nakayama's lemma that $\pi$ is surjective as well. According to lemma~\ref{lem:freemod} the $R$-module $M/W$ is free of rank $n$. Thus, $N\cap W$ is the kernel of the surjective $R$-linear morphism $\pi$ from the free $R$-module $N$ of rank $n+1$ to the free $R$-module $M/W$ of rank $n$. Therefore $N\cap W$ is a direct summand of $N$. The last statement of proposition~\ref{prop:mod} now follows form corollary~\ref{cor:freemod}.
\end{proof}
\begin{corollary}\label{cor:unimod}
	Under the conditions of proposition~\ref{prop:mod} any generating element $w$ of free $R$-module $W\cap N$ of rank $1$ will be unimodular in $M$.
\end{corollary}
\begin{proof}
	It follows from proposition~\ref{prop:mod} that $W\cap N$ is a direct summand of $M$. Therefore according to corollary~\ref{cor:freemod} the element $\overline{w}$ of $k$-linear vector space $\overline{M}$ is not equal to zero. Fix some basis of $M$ that contains $\overline{w}$ and then lift it to $M$ so that $\overline{w}$ is the lift of $w$. The corollary~\ref{cor:unimod} now directly follows from lemma~\ref{lem:freemod}.
\end{proof}

\section{Main result}\label{sec:main}

Here is the main result of this paper:

\begin{theorem}\label{prop:main}
		Let $R$ be a regular local ring of a mixed characteristic $(0,p)$ where $p$ is a prime number other than $2$. Suppose that the quotient ring $R/pR$ is also local and regular. Fix some non-degenerate Pfister form $Q(T_{1},\ldots,T_{2^m})$ over the ring $R$ together with some invertible element $c$ in $R$. Then the equation $Q(T_{1},\ldots,T_{2^m})=c$ has a solution over $R$ if and only if it has a solution over the fraction field $K$ of $R$.
\end{theorem}

The proof of proposition~\ref{prop:main} is heavily based on recent Cesnavicius Theorem \cite{C}.

We will also need the following technical lemma, that was originally proven in \cite{C-T}. Nevertheless, we still decided to give its detailed proof:

\begin{lemma}\label{lem:tech}
	Let $R$ be a local ring with its residue field $k$ of characteristic other than $2$ and let $\alpha_{1},\dots,\alpha_{n}$ be invertible elements of $R$ with $n>2$.
	Then the equation
	
	\begin{equation}\label{eq:sol1}
		\alpha_{1}T_{1}^{2}+\dots+\alpha_{n}T_{n}^{2}=0
	\end{equation}
	has a unimodular solution in  $R^{n}$ if and only if the equation
	\begin{equation}\label{eq:sol2}
		\alpha_{1}T_{1}^{2}+\dots+\alpha_{n-1}T_{n-1}^{2}=-\alpha_{n}
	\end{equation}
	has a solution in $R^{n-1}$.
\end{lemma}
\begin{proof}
	The first implication is trivial: $(v_{1},\dots,v_{n-1})$ being a solution of~\ref{eq:sol2} obviously implies $(v_{1},\dots,v_{n-1},1)$ being a unimodular solution of~\ref{eq:sol1}.
	
	Let  $v:=(v_{1},\dots,v_{n})$ be some unimodular solution of~\ref{eq:sol1}. Note that if element $v_{n}$ does not belong to the maximal ideal of $R$ i.e. is invertible, then
	$$
	\Big(\dfrac{v_{1}}{v_{n}},\dots,\dfrac{v_{n-1}}{v_{n}}\large\Big)
	$$
	is obviously a solution of~\ref{eq:sol2}.
	
	Suppose now that $v_{n}$ is not invertible. As earlier, denote the diagonal quadratic form, corresponding to~\ref{eq:sol1} by $Q$ as well as the associated bilinear form.  Note that for any $u:=(u_{1},\dots,u_{n})$ such that $Q(u,u)$ is invertible in $R$ the result of applying reflection $r_{u}$ to $v$ defines another solution
	$$
	r_{u}(v)= v-\dfrac{Q(u,v)}{Q(u,u)}u\,,
	$$
	of~\ref{eq:sol1}. Moreover, if we assume that $u_{n}$ and $Q(u,v)$ are both invertible in $R$, then $r_{u}(v)$ will also be a solution for~\ref{eq:sol2} since its $n$-th coordinate will also be invertible in $R$.
	Note that $u$ satisfies these assumptions if and only if the corresponding vector $\overline{u}:=(\overline{u}_{1},\ldots,\overline{u}_{n})\in k^n$ satisfies the following inequalities
	$$
	\overline{Q}(\overline{u},\overline{v})\neq0\,,\,\,
	\overline{Q}(\overline{u},\overline{u})\neq0\,,\,\,
	\overline{u}_{n}\neq0\,.
	$$
 Moreover, if there exists such a vector $\overline{u}$ in $k^{n}$ one can just lift it to $R^{n}$ to get a required $u$.
	
	 Since  $v$ is unimodular there exists at least one $i<n$ such that the corresponding coordinate $\overline{v}_{i}$ is not equal to zero. Let us fix such an $i$ together with some $j<n$, other than $i$. Assume that all the coordinates of $\overline{u}$ with indices other than $i$, $j$ and $n$ are equal to zero.
	Then the system of above inequalities will take the following form (for simplicity we will omit the overline symbol):
	$$
	\alpha_{i}v_{i}u_{i}+\alpha_{j}v_{j}u_{j}\neq0\,,\,\,
	\alpha_{i}u^{2}_{i}+\alpha_{j}u^{2}_{j}+\alpha_{n}u^{2}_{n}\neq0\,,\,\,
	u_{n}\neq0\,.
	$$
	Denote $\dfrac{u_{i}}{u_{n}}$ and $\dfrac{u_{j}}{u_{n}}$ by $x$ and $y$ accordingly. Then  the system above can be simplified up to the following form
	
	$$
	\alpha_{i}v_{i}x+\alpha_{j}v_{j}y\neq0\,,\,\,
	\alpha_{i}x^{2}+\alpha_{j}y^{2}\neq-\alpha_{n}\,.
	$$
	Note that if $\alpha_{i}\neq-\alpha_{n}$ then $x=1$ and $y=0$ obviously satisfy this system. On the contrary, if $\alpha_{i}=-\alpha_{n}$ then it is enough to put $x=1$ and find some non-zero $y$ that satisfies the inequality $\alpha_{j}v_{j}y\neq-\alpha_{i}v_{i}$. One can always find such an $y$ since the residue field $k$ containes more than two non-zero elements. This concludes the proof.
\end{proof}

{\it We will now give the proof of theorem~\ref{prop:main}}. 
The first implication is trivial. Assume that the equation $Q(T_{1},\ldots,T_{2^{m}})=c$ has a solution over $K$. 
Then the equation $Q(T_{1},\ldots,T_{2^{m}})-cT_{2^{m}+1}=0$ also has some non-trivial solution $(a_{1},\ldots,a_{2^{m}+1})$ over $K$. 
According to lemma~\ref{lem:tech} the proof of the second implication can be reduced to finding some unimodular solution of the equation 
$Q(T_{1},\ldots,T_{2^{m}})-cT_{2^{m}+1}=0$ over $R$. Denote by $\widetilde{Q}(T_{1},\ldots,T_{2^{m+1}})$ the tensor product of $Q$ and 
the Pfister form $\langle1,-c\rangle$. Then the vector $(a_{1},\ldots,a_{2^{m}+1},0,\ldots,0)$ is a solution of the equation 
$\widetilde{Q}(T_{1},\ldots,T_{2^{m+1}})=0$ over $K$. Since $\widetilde{Q}$ is a Pfister form this means 
that it is hyperbolic over $K$. According to 
Cesnavicius \cite[Corollary 9.6]{C} the form $\widetilde{Q}$ is also hyperbolic over $R$. 
Put $M=R^{2^{m+1}}$ and take a free $R$-submodule $N\subset M$ generated by the first $2^{m}+1$ basis elements of $M$. 
Then the pair $(M,N)$ satisfies the conditions of proposition  ~\ref{prop:mod}. 
In particular, the unimodular isotopic element $w$ from corollary~\ref{cor:unimod} provides a unimodular solution for the equation 
$Q(T_{1},\ldots,T_{2^{m}})-cT_{2^{m}+1}=0$ over $R$. 
{\it This concludes the proof of theorem~\ref{prop:main}.
}

\begin{corollary}\label{cor:main}
	Let $Q(T_{1},\ldots,T_{2^{m}})$ be the Pfister form from theorem~\ref{prop:main}. 
Then the set $S_{Q}$ of all the invertible elements $c\in R^{*}$ that are represented by the form $Q$ is a multiplicative subgroup in $R^{*}$.
	\end{corollary}
\begin{proof}
	Suppose that the elements $c_{1,2}\in R^{*}$ are represented by $Q$, that is the equations $Q=c_{1}$ and $Q=c_{2}$ have some solutions over $R$. Since Pfister forms are mutliplicative over fields this means that the equation $Q=c_{1}c_{2}$ has some solution over $K$. According to theorem~\ref{prop:main} it also has a solution over $R$. Thus $c_{1}c_{2}$ is represented by $Q$ and therefore $S_{Q}$ is closed under multiplication. If there is a solution $v$ for the equation $Q=c$ then $Q(c^{-1}v)=c^{-2}Q(v)=c^{-1}$. Therefore $S_{Q}$ is closed under taking an inverse. Finally since $Q$ is a Phister form the equation $Q=1$ always has a solution. This concludes the proof.
\end{proof}


\begin{thebibliography}{99}
	
	
	\bibitem[B-F/F/P]{B-F/F/P}
\emph{Bayer-Fluckiger E., First U., Parimala R.}
On the Grothendieck--Serre conjecture
for classical groups, (2019). Preprint



\bibitem[C]{C}
\emph{Cesnavicius K.}
Grothendieck-Serre in the quasi-split unramified case, preprint (2021).
Available at https:
//arxiv.org/abs/2009.05299v2.

\bibitem[C-T]{C-T}
\emph{J.-L.Colliot-Th\'el\`ene.}
Formes quadratiques sur les anneaux semi-locaux reguliers.
Colloque sur les Formes Quadratiques, 2 (Montpellier, 1977).\\
Bull. Soc. Math. France Mem. No. 59, (1979), 13--31.



\bibitem[C-T/S]{C-T/S}
\emph{Colliot-Th\'el\`ene J.-L., Sansuc J.-J.}
Cohomologie des groupes de type multiplicatif
sur les sch\'{e}mas r\'{e}guliers.
C. R. Acad. Sci. Paris S\'{e}r. A-B 287.6 (1978), A449--A452.


\bibitem[C-T/O]{C-TO}
\emph{Colliot-Th\'el\`ene J.-L., Ojanguren M.} Espaces Principaux
Homog\`enes Localement Triviaux, Publ.~Math.~IH\'ES 75 (1992),
no.~2, 97--122.


\bibitem[FP]{FP}
\emph{Fedorov R., Panin I.} A proof of Grothendieck--Serre conjecture on principal bundles over a semilocal regular
ring containing an infinite field, Publ. Math. Inst. Hautes Etudes Sci., Vol. 122, 2015, pp. 169--193.

\bibitem[Fe1]{Fe1}
\emph{Fedorov R.}
On the Grothendieck-Serre conjecture on principal bundles in mixed characteristic,
Trans. Amer. Math. Soc., to appear (2021),
arXiv:1501.04224v3

\bibitem[Fe2]{Fe2}
\emph{Fedorov R.}
On the Grothendieck-Serre Conjecture about principal bundles and its generalizations,
Algebra Number Theory, to appear (2021),
arXiv:1810.11844v2

\bibitem[F]{F}
\emph{First U.}
An 8-Periodic Exact Sequence of Witt Groups of Azumaya Algebras with
Involution, (2019). Preprint










\bibitem[OP]{OP} \emph{Ojanguren M., Panin I.}
Rationally trivial hermitian spaces are locally trivial, Math.~Z.
237 (2001), 181--198.






\bibitem[PSV]{PSV} \emph{Panin, I.; Stavrova, A.; Vavilov, N.}
On Grothendieck-Serre's conjecture concerning
principal $G$-bundles over reductive group schemes: I,
Compositio Math. 151 (2015), 535--567.



\bibitem[P0]{P0}
\emph{Panin, I.} Rationally isotropic quadratic spaces are locally
isotropic. Invent. math. 176, 397--403 (2009).



\bibitem[P1]{P1}
\emph{Panin, I.}
Nice triples and the Grothendieck-Serre conjecture concerning principal
G-bundles over reductive group schemes,
Duke Math. J. , 2019, 168:2, 351--375;


\bibitem[P2]{P2}
\emph{Panin, I.}
Two purity theorems and the Grothendieck-Serre
conjecture concerning principal G-bundles,
Sbornik: Mathematics, 2020, 211:12, 1777--1794.

\bibitem[P3]{P3}
\emph{Panin, I.} Proof of the Grothendieck--Serre conjecture on principal
bundles over regular local rings containing a field,
Izvestiya: Mathematics, 2020, 84:4, 780--795.

\bibitem[P4]{P4}
\emph{Panin, I.} Moving lemmas in mixed characteristic and applications.
arxiv.org, 2202.00896v1.

\bibitem[P5]{P5}
\emph{Panin, I.} On Grothendieck–Serre conjecture in mixed characteristic for SL1,D.
arxiv.org, 2202.05493v1.


\bibitem[PP1]{PP1}
\emph{Panin, I., Pimenov K.} Rationally isotropic quadratic spaces
are locally isotropic:II. Documenta Mathematica, Extra
Volume: Andrei A. Suslin's Sixtieth Birthday, 515--523,
(2010).

\bibitem[PP2]{PP2}
\emph{Panin, I., Pimenov K.} Rationally isotropic quadratic spaces are locally isotropic: III.
St. Petersburg Math. J., 2015, 27:6,  234--241

\bibitem[PS]{PS}
\emph{Panin, I.; Suslin, A.}
On a conjecture of Grothendieck concerning Azumaya algebras,
St.Petersburg Math. J., 9, no.4, 851--858 (1998)



\bibitem[Po]{Po}
\emph{Popescu, D.} General N\'eron desingularization and
approximation, Nagoya Math.~J. 104 (1986), 85--115.






\bibitem[Sc]{Sc}
\emph{S.~Scully} The Artin--Springer theorem for quadratic forms over semi-local rings with finite residue fields.
Proc. Amer. Math. Soc. 146 (2018), 1--13.


	

	
	
\end{thebibliography}
\end{document}